\documentclass[12pt]{amsart}
\usepackage{amsmath}
\usepackage{amsxtra}
\usepackage{amstext}
\usepackage{amssymb}
\usepackage{amsthm}
\usepackage{xcolor}
\usepackage{dsfont}

\newtheorem{thm}{Theorem}[section]
\newtheorem{prop}[thm]{Proposition}
\newtheorem{lemma}[thm]{Lemma}
\newtheorem{cor}[thm]{Corollary}
\newtheorem{cla}[thm]{Claim}
\theoremstyle{definition}
\newtheorem{definition}[thm]{Definition}
\newtheorem{assump}[thm]{Assumptions}

\theoremstyle{remark}
\newtheorem{remark}[thm]{Remark}
\numberwithin{equation}{section}
\DeclareMathOperator{\dist}{\operatorname{dist}} 
\DeclareMathOperator{\supp}{\operatorname{supp}}

\DeclareMathOperator{\lip}{\operatorname{Lip}}
\def\R{\mathbb{R}}
\def\eps{\varepsilon}

\def\wt{\widetilde}

\def\Xint#1{\mathchoice
{\XXint\displaystyle\textstyle{#1}}%
{\XXint\textstyle\scriptstyle{#1}}%
{\XXint\scriptstyle\scriptscriptstyle{#1}}%
{\XXint\scriptscriptstyle%
\scriptscriptstyle{#1}}%
\!\int}
\def\XXint#1#2#3{{\setbox0=\hbox{$#1{#2#3}{%
\int}$ }
\vcenter{\hbox{$#2#3$ }}\kern-.6\wd0}}

\def\dashint{\Xint-}
\def\I{\mathcal{I}}
\def\R1{\widetilde{R}}
\def\T1{\widetilde{T}}
\def\dist{\operatorname{dist}}
\def\supp{\operatorname{supp}}
\def\Lip{\operatorname{Lip}}
\def\eps{\varepsilon}
\def\kap{\varkappa}
\def\R{\mathbb{R}}

\def\wt{\widetilde}

\def\S{\mathcal{S}}
\def\I{\mathcal{I}}
\def\kap{\varkappa}

\def\F{\mathcal{F}}
\def\I{\mathcal{I}}
\def\SB{C_{\text{sb}}}
\def\SK{\mathbb{S}_{s}}
\def\Ms{\mathcal{M}_s}
\def\alphaflat{\alpha_{\mu,s}^{\operatorname{flat}}}
\def\alphaspike{\alpha^{\operatorname{Spike}_k}}

\def\Xint#1{\mathchoice
   {\XXint\displaystyle\textstyle{#1}}%
   {\XXint\textstyle\scriptstyle{#1}}%
   {\XXint\scriptstyle\scriptscriptstyle{#1}}%
   {\XXint\scriptscriptstyle\scriptscriptstyle{#1}}%
   \!\int}
\def\XXint#1#2#3{{\setbox0=\hbox{$#1{#2#3}{\int}$}
     \vcenter{\hbox{$#2#3$}}\kern-.5\wd0}}

\def\dashint{\Xint-}

\author[B. Jaye]{Benjamin Jaye}

\address{School of Mathematical Sciences, Clemson University, Clemson, SC 29631, USA}

\email{bjaye@clemson.edu}

\author[T. Merch\'an]{Tom\'as Merch\'an}

\address{Department of Mathematical Sciences, Kent State University, Kent, Ohio 44240, USA}

\email{tmercha2@kent.edu}

\thanks{Research supported in part by NSF DMS-1830128 and DMS-1800015.}
\title[Principal Value Integrals ]{On the problem of existence in principal value of a Calder\'{o}n-Zygmund operator on a space of non-homogeneous type}
\date{\today}

\begin{document}

\begin{abstract}In this paper we study the relationship between two fundamental regularity properties of an $s$-dimensional Calder\'{o}n-Zygmund operator (CZO) acting on a Borel measure $\mu$ in $\R^d$, with $s\in (0,d)$.

In the classical case when $s=d$ and $\mu$ is equal to the Lebesgue measure, Calder\'on and Zygmund showed that if a CZO is bounded in $L^2$ then the principal value integral exists almost everywhere.  However, there are by now several examples showing that this implication may fail for lower-dimensional kernels and measures, even when the CZO has a homogeneous kernel consisting of spherical harmonics.

We introduce sharp geometric conditions on $\mu$, in terms of certain scaled transportation distances, which ensure that an extension of the Calder\'{o}n-Zygmund theorem holds.  These conditions are necessary and sufficient in the cases of the Riesz transform and the Huovinen transform.  Our techniques build upon prior work by Mattila and Verdera, and incorporate the machinery of symmetric measures, introduced to the area by Mattila.
 \end{abstract}
\maketitle

\section{Introduction}

In this paper we introduce sharp sufficient conditions on a (locally finite, non-negative Borel) measure $\mu$ which ensure that if a Calder\'{o}n-Zygmund operator (CZO) is bounded with respect to $L^2(\mu)$, then the operator exists in the sense of principal value.  We will be working with the following class of (particularly smooth) kernels.
\begin{definition}\label{CZdef} Fix $s \in (0,d)$. A function $K: \mathbb{R}^d  \setminus \{0\} \rightarrow \mathbb{C}^d$  is an $s$-dimensional Calder\'on-Zygmund kernel if there is a constant $C_K>0$ such that following properties are satisfied for every $x, x' \in \R^d\backslash \{0\}$:
\begin{enumerate}
     \item\label{size}  $|K(x)| \leq \frac{C_K}{|x|^s}$,
     \item\label{odd} $K(-x)=-K(x)$, and
     \item\label{decay} if $|x-x'| \leq \frac{|x|}{2}$ then $|K(x)- K(x')|  \leq \frac{C_K|x-x'| }{|x|^{s+1} }$.\end{enumerate}
\end{definition}
  Let $\mu$ be a measure.  We say that the Calder\'{o}n-Zygmund operator $T$ associated to $K$ is \emph{bounded in $L^2(\mu)$} if there is a constant $C>0$ such that
\begin{equation}\label{l2bd}\sup_{\eps>0}\int_{\R^d}\Bigl|\int_{\R^d\backslash B(x,\eps)}K(x-y)f(y)d\mu(y)\Bigl|^2d\mu(x)\leq C\|f\|_{L^2(\mu)}^2
\end{equation}
for every $f\in L^2(\mu)$.  The least constant $C>0$ for which (\ref{l2bd}) holds for all $f\in L^2(\mu)$ is called the norm of $T$.\\

On the other hand, the CZO $T$ \emph{exists in the sense of principal value} if for every complex measure $\nu$,
\begin{equation}\label{PV}\lim_{r\to 0}\int_{|x-y|>r}K(x-y)d\nu(y)\text{ exists for }\mu\text{-almost every }x\in \R^d.
\end{equation}

For classical CZOs ($s=d$) acting in Euclidean space $\R^d$ with $\mu=m_d$ (the Lebesgue measure), a density argument ensures that the boundedness of a CZO in $L^2(m_d)$ implies the existence of the CZO in the sense of principal value $m_d$-almost everywhere; see for instance \cite{CZ, SW}.  \\

However, there are by now several examples which show that the Calder\'{o}n-Zygmund theorem does not necessarily extend when the Lebesgue measure is changed to a different underlying measure, see e.g. \cite{CH, Dav}. It was shown in \cite{JN3} that there is a measure $\mu$ satisfying $\mu(B(x,r))\leq r$ for every disc $B(x,r)\subset \mathbb{C}\cong\R^2$ such that the one-dimensional CZO associated to the Huovinen kernel $K(z) = \frac{z^k}{|z|^{k+1}}$, $k\geq 3$ odd, is bounded in $L^2(\mu)$ but the principal value integral fails to exist $\mu$-almost everywhere.  Huovinen \cite{H} has previously studied the geometric consequences of the existence of the principal value integral associated to this kernel, which plays a significant role in the literature due to being the prototypical example of a one-dimensional CZ kernel in the plane for which the Melnikov-Menger curvature formula (see e.g. \cite{MMV}) fails to hold, see the survey papers \cite{M2, M3}.  \\

Notwithstanding these examples, it is expected that an analogue of the classical Calder\'{o}n-Zygmund theorem should hold for the $s$-Riesz transform, the CZO with kernel $K(x)= \frac{x}{|x|^{s+1}}$ ($x\in \R^d$).  Indeed, a long standing conjecture\footnote{Often referred to as a variant of the David-Semmes question \cite{DS}.} states that if $\mu$ is a non-atomic measure, then whenever the $s$-Riesz transform operator is bounded in $L^2(\mu)$, it also exists in principal value.  This was proved for $s=1$ by Tolsa\footnote{Tolsa only considered the case $d=2$, but the method extends.} (see \cite{To6}), and for $s=(d-1)$, where it can be proved by combining the deep results of Eiderman-Nazarov-Volberg \cite{ENV}, Nazarov-Tolsa-Volberg \cite{NToV}, and Mattila-Verdera  (stated as Theorem \ref{MVthm} below) \cite{MV}.  It is an open problem for $s=2,\dots, d-2$.\\

The results described in the preceding paragraphs combine to show that the problem of when (\ref{l2bd}) implies (\ref{PV}) depends quite subtly on the algebraic structure of the underlying kernel in the operator, and the purpose of this paper is to develop some theory to better understand this.\\

  In general, the existence of the principal value integral should be viewed as stronger (but more qualitative) than the $L^2$ boundedness of the associated singular integral operator. Indeed, Nazarov-Treil-Volberg (\cite{NTV2}, see also  Tolsa \cite{To3}), proved the following theorem:
  \begin{thm}
Let $\mu$ be a measure with finite upper density, i.e. $$\overline{D}_{\mu,s}(x):=\limsup_{r\to 0}\frac{\mu(B(x,r))}{r^s}<\infty$$ for $\mu$-almost every $x\in \R^d$, and satisfying  (\ref{PV}) with $\nu=\mu$. Then for every $\eps>0$, there exists a set $E_{\eps}$ with $\mu(\R^d\backslash E_{\eps})<\eps$ such that $T$ is bounded in $L^2(\mu_{|E_\eps})$ with norm depending on $\eps$.
   \end{thm}


\subsection{A sharp sufficient condition for the existence of principal values}  An important result relating $L^2$ boundedness to the existence of principal values is the following theorem of Mattila and Verdera\footnote{In \cite{MV}, the theorem is stated in the generality of a large class of metric spaces.} \cite{MV}.  Set $\Ms$ to be the collection of measures $\mu$ satisfying
  $$\mu(B(x,r))\leq r^s \text{ for every }x\in \R^d\text{ and }r>0. $$
  \begin{thm}[The Mattila-Verdera Theorem]\label{MVthm}  Fix $\mu\in \Ms$.  Suppose that a CZO $T$ is bounded in $L^2(\mu)$, and $\mu$ has zero $s$-density in the sense that $$\overline{D}_{\mu,s}(x):=\limsup_{r\to 0}\frac{\mu(B(x,r))}{r^s}=0$$ for $\mu$-almost every $x\in \R^d$. Then $T$ exists in principal value (the limit (\ref{PV}) holds for every complex measure $\nu$).
\end{thm}
 The zero density condition in the Mattila-Verdera theorem is necessary for the principal value integral to exist for the $s$-Riesz transform if $s\notin \mathbb{Z}$ (see Theorem \ref{firstpaperthm} below), but is not sharp if $s\in \mathbb{Z}$ (for instance consider the $s$-dimensional Hausdorff measure restricted to an $s$-plane).\\

Our first main result is a natural strengthening of the Mattila-Verdera theorem for CZOs of integer dimension that gives a necessary and sufficient condition in the case of the $s$-Riesz transform.  To state the theorem, we will require the introduction of (a variant of) the Lipschitz transportation number. Variants of transportation numbers have proved very useful in the geometric study of singular integral operators, following their introduction to the area by Tolsa \cite{To1,To2}.\\

Fix $s\in \mathbb{Z}$.  Set $\mathcal{G}(s,d)$ to be the collection of $s$-dimensional linear subspaces of $\R^d$.  For $x\in \R^d$ and $r>0$ we define the transportation distance from $\mu$ to the affine plane $x+L$, $L\in \mathcal{G}(s,d)$, by
 $$\alpha_{\mu, L}(B(x,r))=\sup_{\substack{f\in \Lip_0(B(x,4r))\\\|f\|_{\Lip}\leq \tfrac{1}{r}}}\Bigl|\frac{1}{r^s}\int_{\R^d} \varphi\Bigl(\frac{|\cdot-x|}{r}\Bigl) f d(\mu-c_{\mu,L}\mathcal{H}^s_{x+L})\Bigl|,
 $$
 where
 \begin{itemize}
 \item we denote by $\mathcal{H}^s_{x+L}$ the $s$-dimensional Hausdorff measure restricted to the affine plane $x+L$,
 \item  $\varphi$ is a smooth function that satisfies $\varphi\equiv 1$ on $(0,3)$ and $\supp(\varphi) \subset (0,4)$,
 \item $\Lip_0(B(x,4r))$ denotes the collection of Lipschitz continuous functions compactly supported in $B(x,4r)$ (see also Section \ref{prem}),
 \item and $$c_{\mu,L}= \int_{\R^d} \varphi\bigl(\tfrac{|\cdot-x|}{r}\bigl) \,d\mu\Bigl[\int_{\R^d} \varphi \bigl(\tfrac{|\cdot-x|}{r}\bigl) \,d\mathcal{H}^s_{x+L}\Bigl]^{-1}.$$
 \end{itemize}
The number $c_{\mu, L}$ of course depends on $x$ and $r$, but we will only consider it at a fixed scale at any given time so we suppress this dependence.

We then define the transportation distance to affine $s$-planes by $$\alphaflat(B(x,r)) = \inf_{L\in \mathcal{G}(s,d)}\alpha_{\mu, L}(B(x,r)).
$$
Observe that $\alphaflat(B(x,r))$ can be small if either the density $\frac{\mu(B(x, 4r))}{r^s}$ is small, or if $\mu$ is well approximated by an $s$-plane within the ball $B(x,4r)$.  With this notation we have the following theorem:

\begin{thm}\label{planethm}  Fix $s\in \mathbb{Z}$ and $\mu\in \Ms$.  Suppose that $T$ is bounded in $L^2(\mu)$, and \begin{equation}\label{flatcond}\lim_{r\to 0}\alphaflat(B(x,r))=0\text{ for }\mu\text{-almost every }x\in \R^d,\end{equation} then $T$ exists in principal value.
\end{thm}

  Theorems \ref{MVthm} and \ref{planethm} are both sharp for the $s$-Riesz kernel $K(x)=\frac{x}{|x|^{s+1}}$.  Indeed, we recall the following theorem, which is a consequence of results by Tolsa \cite{To4} and Ruiz de Villa-Tolsa \cite{RVT}.  For a relatively simple direct proof see Theorem 1.2 of \cite{JM}.

\begin{thm}\label{firstpaperthm}  Fix $\mu$ with $\overline{D}_{\mu,s}(x)<\infty$ for $\mu$-almost every $x\in \R^d$.  Suppose that the CZO associated to the $s$-Riesz kernel exists in principal value.  Then
\begin{enumerate}
 \item $s\not\in \mathbb{Z}$ and $\mu$ has zero density, or
 \item $s\in \mathbb{Z}$ and $\mu$ satisfies $\lim_{r\to 0}\alphaflat(B(x,r)=0$  for $\mu$-a.e $x\in \R^d$.
 \end{enumerate}
 \end{thm}

\begin{remark} It is well-known, see e.g. \cite{MM, Ver}, that if a measure $\mu$ is $s$-rectifiable\footnote{that is, $\mu $ is absolutely continuous with respect to $\mathcal{H}^s$, and its support is contained in the union of a countable number of $s$-dimensional Lipschitz (or $C^1$) submanifolds, up to an exceptional set of $s$-dimensional Hausdorff measure zero}, then any associated $s$-dimensional CZO $T$ exists in principal value integral.  However, there are many examples of measures $\mu$, see e.g. Section 5.8 of \cite{P}, whose support has locally finite $s$-dimensional measure and the condition (\ref{flatcond}) holds, but $\mu$ is not rectifiable.\end{remark}

Theorem \ref{firstpaperthm} shows that the additional condition placed on the measure in Theorems \ref{MVthm} and \ref{planethm} are necessary conditions for the $s$-Riesz transform to exist in principal value\footnote{However, it is not known whether (\ref{flatcond}) already follows from the $L^2(\mu)$ boundedness of the $s$-Riesz transform if $s=2,\dots, d-2$, which is (a generalization of) the aforementioned question of David-Semmes \cite{DS}.}.  Consequently they are the sharp conditions to consider when working with a large class of operators.  However, there are other CZOs of particular interest for which our method enables us to provide more information.  We illustrate this with the case of the Huovinen kernels.

\subsection{On the Huovinen kernels} A second goal of this work is to answer a question left open by the work \cite{JN3} and identify the geometric condition responsible for the difference between $L^2$ boundedness and existence of principle value for the CZO associated to the Huovinen kernel $K_k(z)=\frac{1}{|z|}\bigl(\frac{z}{|z|}\bigl)^k$, $k$ odd. 

  A \emph{$k$-spike measure} associated to a subspace $L\in \mathcal{G}(1,2)$, and $\omega \in \mathbb{C}$, takes the form
\begin{equation}\nonumber \nu_{m,L,\omega}=\sum_{n=0}^{m-1} \mathcal{H}_{e^{\pi in/m}L+\omega}, \text{ where }m\in \mathbb{N}\text{ divides }k\, (\text{henceforth } m\mid k).
\end{equation}
We set $\text{Spike}_k$ to be the collection of all such $k$-spike measures over $L\in \mathcal{G}(1,2)$, $\omega\in \mathbb{C}$, and $m\mid k$.

For $z\in \mathbb{C}$, $r>0$, we proceed to define the modified transportation number $\alphaspike_{\mu}(B(z,r))$ by
$$\alphaspike_{\mu}(B(z,r)) =\inf_{\substack{\nu \in \text{Spike}_k
\\z \in \supp(\nu)}}\sup_{\substack{f\in \Lip_0(B(z,4r))\\\|f\|_{\Lip}\leq \tfrac{1}{r}}}\Bigl|\frac{1}{r}\int_{B(z,r)}\varphi\Bigl(\frac{|\cdot-z|}{r}\Bigl) f \;d(\mu-c_{\mu,\nu}\nu)\Bigl|,
$$
where $$
c_{\mu,\nu}= \int_{\R^d} \varphi\bigl(\tfrac{|\cdot-z|}{r}\bigl) \,d\mu\Bigl[\int_{\R^d} \varphi \bigl(\tfrac{|\cdot-z|}{r}\bigl) \,d\nu\Bigl]^{-1}.$$

\begin{thm} \label{HThm}  Fix $k$ odd and $\mu\in \mathcal{M}_1$.  Suppose that the CZO $T$ associated to the Huovinen kernel $K_k$ is bounded in $L^2(\mu)$.  Then $T$ exists in principal value if and only if $\lim_{r\to 0}\alphaspike_{\mu}(B(z,r))=0$ for $\mu$-almost every $z\in \mathbb{C}$.
\end{thm}
The `only if' direction of this theorem was shown in Theorem 1.5 of \cite{JM}, which in turn was based upon Huovinen's work \cite{H}.

The example \cite{JN3} shows that $L^2(\mu)$-boundedness cannot imply by itself the existence of $T$ in the sense of principal value, so $L^2(\mu)$-boundedness does not imply the property $\lim_{r\to 0}\alphaspike_{\mu}(B(z,r))=0$ for $\mu$-almost every $z\in \mathbb{C}$.\\

Both Theorems \ref{planethm} and \ref{HThm} follow from a more general statement Theorem \ref{GenThm} below which relates the existence of principal value to the transportation distance to a certain collection of \emph{symmetric measures}\footnote{Introduced by Mattila \cite{M}.} associated to the operator, see Sections \ref{genthmsec} and \ref{applications}. 
Theorem \ref{GenThm} is proved using the same basic scheme as the one followed by Mattila-Verdera to prove Theorem \ref{MVthm} (and we recover this theorem as a simple special case of Theorem \ref{GenThm}), but significant modifications are required, since dealing with transportation numbers introduces new geometric situations that do not arise in the set-up of \cite{MV}.

\subsection*{Acknowledgement}  We thank the referee for carefully reading the paper and making several remarks which have helped the presentation of the results.

\section{Preliminaries and Notation} \label{prem}
In this section we collect some definitions and preliminaries.
\begin{enumerate}
\item[\textbullet]  We shall denote by $C>0$ and $c>0$ respectively large and small constants that may change from line to line and can depend on $d$, $s$, the constant $C_K$ from Definition \ref{CZdef} and the quantities $\Theta$ and $\SB$ introduced in Assumptions \ref{SKdef} in Section \ref{genthmsec}.  By $A\lesssim B$, we shall mean that $A\leq CB$ for some constant $C>0$.  By $A\ll B$ we shall mean that $A\leq c_0B$ for some sufficiently small constant $c_0>0$ depending on $d,s,C_K$, $\SB$ and $\Theta$.
\item[\textbullet] We shall denote $\mathcal{G}(s,d)$ as the collection of s-dimensional linear subspaces of $\R^d$.
	\item[\textbullet] For $x \in \mathbb{R}^d$
	and $r>0$, $B(x,r)$ denotes the open ball centered at $x$ with radius $r$.
\item[\textbullet] For a function $f$ defined on an open set $U\subset \R^d$, define
$$\|f\|_{\Lip(U)} = \sup_{x,y\in U, \, x\neq y}\frac{|f(x)-f(y)|}{|x-y|}.
$$
In the case $U=\R^d$, we write $\|f\|_{\Lip}$ instead of $\|f\|_{\Lip(\R^d)}$.
\item[\textbullet] For an open set $U\subset \R^d$, define $\Lip_0(U)$ to be the collection of functions $f$ supported on a compact subset  of $U$ with $\|f\|_{\Lip(U)}<\infty$.
	\item[\textbullet] We denote by $\supp(\mu)$ the closed support of the measure $\mu$; that is,
	$$ \supp(\mu) = \mathbb{R}^d \setminus \{ \cup B: B \text{ is an   open ball with }\mu(B)=0 \}.$$
        	\item[\textbullet]  We denote by $\Ms$ the collection of non-negative measures $\mu$ satisfying the growth bound $\mu(B(x,r))\leq r^s$ for all $x\in \R^d$ and $r>0$.
        	\item[\textbullet] \label{etadef} For $\kap, r>0$ and $x\in \R^d$, set $\eta_{\kap,r,x}$ to be a non-negative radial function satisfying $\eta_{\kap, r,x}\equiv 1$ on $B(x,r)$, $\eta_{\kap, r,x} \equiv 0$ on $\R^d\backslash B(x,(1+\kap)r)$, and $\|\eta_{\kap,r,x}\|_{\Lip}\leq \frac{1}{\kap r}$.
     \item[\textbullet] We introduce the bump function $\varphi: [0,\infty) \mapsto [0,\infty)$, satisfying $\varphi \in C^\infty$, $\varphi \equiv 1$ on $(0,3)$, and $\supp \varphi \subset (0,4)$.
     \item[\textbullet] Given a Borel measure $\mu$, for any $x \in \mathbb{R}^d$ and $r >0$, we set
     $$D_\mu(B(x,r))=\frac{\mu(B(x,r))}{r^s},  \text{ and}$$
  $$\mathcal{I}_\mu(B(x,r)) = \int_{\mathbb{R}^d} \varphi\left(\frac{|x-y|}{r} \right) \,d\mu(y).$$
\item[\textbullet] Given a complex measure $\sigma$, we also define the truncated operators $T_r$:
$$ T_r (\sigma) (x)= \int_{\R^d \setminus B(x,r)} K(x-y) \,d\sigma(y),  \, x \in \R^d.$$
\item[\textbullet] Given a measure $\mu$ and $f \in L^1_{loc}(\mu)$, we define the average
$$\dashint_E fd\mu = \frac{1}{\mu(E)}\int_E fd\mu.$$

\end{enumerate}

\section{The general theorem} \label{genthmsec} We now move onto stating our general result.
Fix $s\in (0,d)$ (integer or not), and fix an $s$-dimensional Calder\'{o}n-Zygmund kernel $K$.

\subsection{Symmetric measures}  Two notions of symmetry of a measure will play a role in our results.  The first is Mattila's notion of a symmetric measure.

\begin{definition}[Symmetric measure] Let $\nu$ be a measure.

\begin{itemize}
\item A point $x\in \R^d$ is a  $K$-symmetric point for $\nu$, written $x\in \mathcal{S}(K,\nu):=\mathcal{S}(\nu)$, if
$$\int_{B(x,r)} K(x-y)|x-y|^s d\nu(y)=0 \text{ for every }r>0.
$$
\item The measure $\nu$ is a $K$-symmetric measure  if $\supp(\nu)\subset \mathcal{S}(\nu)$.
\item Set $\mathcal{S}_{K,s}:=\mathcal{S}_s = \{ \nu: \nu \text{ is }K\text{-symmetric}, \; \nu\in \mathcal{M}_s\}.$
\end{itemize}\end{definition}

We shall also require a local notion of reflection symmetry, defined as follows.

\begin{definition}[Reflection symmetric measure]  A measure $\nu$ is reflection symmetric about $z$ in a ball $B$ if, whenever $E\subset B$ is a Borel set, then $\nu(2z-E)=\nu(E)$.\end{definition}

We can now introduce our main assumptions required to state the general theorem.

 \begin{assump}\label{SKdef}
 Set $\SK$ to be any subset of $\mathcal{S}_{s}$ with the following properties:
\begin{enumerate}
\item \label{SymSmallBdary}(Small Boundaries) There exists $\SB >0$ such that whenever $\nu\in \SK$, $x\in \R^d$, and $r>0$ satisfy $B(x,r/2)\cap \supp(\nu)\neq \varnothing$, then for every $\tau\in (0,1]$,
    $$\nu(B(x, (1+\tau)r)\backslash B(x,r))\leq \SB\tau\nu(B(x,r)).
    $$
\item\label{SymReflSym} (Nearby Points of Reflection Symmetry) There exists $\Theta \geq 1$ such that for every $\nu\in \SK$, $r>0$, and $x\in \supp(\nu)$, the following alternative holds:
    \begin{enumerate}
    \item Either $\nu$ is reflection symmetric about $x$ in $B(x, r)$, or
    \item There exists $\wt{x}$ in $\supp(\nu)\cap B(x,\Theta r)$ such that $\nu$ is reflection symmetric about $\wt{x}$ in $B(\wt{x},64\Theta r)$.
    \end{enumerate}
\end{enumerate}
\end{assump}
For a measure $\mu$, we define the transportation distance to a measure $\nu$ on the scale $B(x,r)$ by
$$\alpha_{\mu,\nu}(B(x,r))=\sup_{f\in \F_{x,r}}\Bigl|\frac{1}{r^s}\int_{\R^d}\varphi\Bigl(\frac{|x-y|}{r}\Bigl)f(y)d[\mu-c_{\mu, \nu}\nu](y)\Bigl|,$$
where $\F_{x,r} = \{f\,\in\, \lip_0(B(x,4r)),\;\|f\|_{\lip}\leq \frac{1}{r}\}$,
and
$$c_{\mu, \nu}(x,r) :=c_{\mu, \nu}= \begin{cases}\displaystyle \frac{\I_{\mu}(B(x,r))}{\I_{\nu}(B(x,r))} \text{ if }\I_{\nu}(B(x,r))>0,\\
\;\;\;\;0\;\;\;\;\;\;\;\;\;\;\;\;\text{ if }\I_{\nu}(B(x,r))=0.\end{cases}$$
The transportation distance of $\mu$ to the class $\SK$ on the scale $B(x,r)$ is then given by
$$\alpha_{\mu,\SK}(B(x,r)) = \inf_{\substack{\nu\in \SK:\, x\in \mathcal{S}(\nu)\\ \supp(\nu)\cap B(x,r/8)\neq\varnothing}} \alpha_{\mu,\nu}(B(x,r)).
$$

We now proceed to state the main theorem of the paper.
\begin{thm}\label{GenThm}  Fix $\mu\in \Ms$.  Suppose that the $s$-dimensional CZO associated to $K$ is bounded in $L^2(\mu)$.  Assume also that $$\lim_{r\to 0}\alpha_{\mu,\SK}(B(x,r))=0$$ for $\mu$-almost every $x\in \R^d$.  Then for every finite complex measure $\sigma$,
$$\lim_{r\to 0}\int_{|x-y|>r}K(x-y)d\sigma(y) \text{ exists for }\mu\text{-almost every }x\in \R^d.
$$
\end{thm}

\begin{remark}  Theorem \ref{GenThm} leads to the natural question of whether the Assumptions \ref{SKdef} are necessary.  It would certainly be a very nice result if one could show that, in Theorem \ref{GenThm}, one could take $\SK$ to be the whole class $\mathcal{S}_s$ of symmetric measures.  We do not have a counterexample to show this cannot be true.
\end{remark}

\subsection{Scheme of proof}
Before finishing this section, we comment on the scheme of the proof.
In Section \ref{reduction} we recall a well-known reduction that it suffices to only prove the existence of the limit (\ref{PV}) with $\nu=\mu$. Our goal will be to prove that the principal value integral $\lim_{r\to 0}T_r(\mu)$  equals the Mattila-Verdera weak limit function $\widetilde{T}(\mu)(x) = \lim_{r\to 0}\frac{1}{\mu(B(x,r))}\int_{B(x,r)}T(\chi_{\R^d\backslash B(x,r)}\mu)d\mu$ for $\mu$-almost every $x\in \R^d$ (the existence of the Mattila-Verdera weak limit function is ensured by the $L^2(\mu)$ boundedness of the operator, see Theorem \ref{Weak}). To implement this idea will require a careful study of several truncated integrals where the geometry imposed by the condition $\lim_{r \to 0} \alpha_{\mu,\SK}(B(x,r))=0$ will be essential (see Section \ref{estA}).

Observe that the transportation coefficient $\alpha_{\mu, \mathbb{S}_s}(B(x,r))$ is small in two scenarios.  Either
\begin{itemize}
\item the density $\frac{\mu(B(x,4r))}{r^s}$ is small, or
\item $\mu$ inherits (to some extent) the geometric structure of some symmetric measure $\nu\in \mathbb{S}_s$ in the ball $B(x,4r)$.
\end{itemize}
We will often argue via this alternative in what follows, as we may have that $\lim_{r\to 0}\alpha_{\mu,\SK}(B(x,r))=0$ as a result of alternating between these two scenarios.  In the case of the Mattila-Verdera theorem (Theorem \ref{MVthm}), only the first scenario occurs.

It is worth remarking that, if one only wants to prove Theorem \ref{planethm}, then the proof that follows can be simplified in several places. However the proof of Theorem \ref{HThm} requires a more careful study due to the fact that the symmetric measures in this case (spike measures) are not reflection symmetric at every point and scale.

As is common in analysis on non-homogenous spaces, we will look to carry out our analysis on doubling scales (where $\mu(B(x,Ar))$ is not much larger than $\mu(B(x,r))$, for some fixed $A$).  We shall therefore revisit some of the basic tools introduced by David-Mattila \cite{Dav3, DM}.

\section{Main applications}\label{applications}
In this section we study two instances in which the general theorem Theorem \ref{GenThm} can be applied, resulting in recovering Theorem \ref{MVthm} as well as proving Theorems \ref{planethm} and \ref{HThm}.

\subsection{Theorem \ref{MVthm}} We first remark that in order to recover the Mattila-Verdera theorem (Theorem \ref{MVthm}), we consider the case when the collection $\SK$ consists of the zero measure. Then $\lim_{r\to 0}\alpha_{\mu,\SK}(B(x,r))=0$ at $x\in \R^d$ if and only if $\overline{D}_{\mu,s}(x)=0$. In fact, in this case, our proof essentially collapses to the proof found in \cite{MV}.

\subsection{Theorem \ref{planethm}} Fix $s \in \mathbb{Z}$ with $s\in (0,d)$, and  let $$\SK=\{\nu = c\mathcal{H}^s_L:  L \in \mathcal{G}(s,d), \, c>0,\, \nu\in \Ms\}.$$  It is clear that for any $s$-dimensional CZ kernel, every measure in $\SK$ is a symmetric measure.
Therefore, if $\alphaflat(B(x,r)) \to 0$ as $r \to 0$, then $\alpha_{\mu,\SK}(B(x,r)) \to 0$ as $r \to 0$.

Every measure in $\SK$ fulfils the properties concerning power growth, small boundaries and nearby points of reflection symmetry from Assumptions \ref{SKdef}. For the property of nearby points of reflection symmetry, notice that every point in the $s$-plane is a reflection symmetric point and hence satisfies part (a) of the definition.  Therefore Theorem \ref{GenThm} is applicable and Theorem \ref{planethm} follows.

\subsection{Theorem \ref{HThm}}In the case of Huovinen kernels $K_k(z)=\frac{z^k}{|z|^{k+1}}$ where $k$ is odd, we set $\mathbb{S}_{1}=\mathcal{S}_{K_k,1}$ and we recall the following theorem, which is essentially due to Huovinen \cite{H} (see Theorem 1.5 of \cite{JM}).

\begin{thm} If $s=1$ then $$\mathcal{S}_{K_k,1}= \{c\nu: c\geq 0, \,\nu\in \text{Spike}_k, \, c\nu\in \mathcal{M}_1\}$$ and for $\nu \in \text{Spike}_k$, $\mathcal{S}(\nu)=\supp(\nu)$.
\end{thm}

As a consequence of the theorem, we have that $\alpha_{\mu, \mathbb{S}_1}(B(z,r)) = \alphaspike_{\mu}(B(z,r))$ for $z\in \mathbb{C}$ and $r>0$.

In \cite{JM}, it was proved that the existence of $T$ in the sense of principal value implies that $\lim_{r \to 0} \alpha_{\mu, \mathcal{S}_s}(B(z,r)) = 0$ for $\mu$-a.e. $z \in \mathbb{C}$ (see Theorem 1.4 and Proposition A.1 in \cite{JM}).

To conclude Theorem \ref{HThm} we need to verify that Assumptions \ref{SKdef} are satisfied for the collection of measures $\mathbb{S}_1=\mathcal{S}_{K_k,1}$.  It is clear that such measures satisfy the small boundaries condition. Regarding the nearby points of reflection symmetry property, let $\nu=c\sum_{n=0}^{m-1} \mathcal{H}_{e^{\pi in/m}L+z} \in \mathcal{S}_{K,1} $ for some $m\mid k$, $L \in \mathcal{G}(1,2)$ and $z \in \mathbb{C}$; and  let $y \in \supp(\nu)$. If $\nu$ is not reflection symmetric about $y$ in $B(y,r)$, then we have that $|y-z| \leq \frac{r}{\sin(\pi/k)}$ ($m\leq k$) and $\nu$ is reflection symmetric about $z$ in $B(z,R)$ for every $R>0$, hence we may choose $\Theta=\frac{2}{\sin(\pi/k)}$.

\section{Reduction to the case $\nu =\mu$ in (\ref{PV})}\label{reduction}
We next record a standard result, whose proof may be found in Section 8.2.1. of Tolsa \cite{To5}.
\begin{thm}
Fix $\mu\in \Ms$, and a CZO kernel $K$.  Suppose that the CZO $T$ associated to $K$ is bounded in $L^2(\mu)$. If $(\ref{PV})$ holds with $\nu=\mu$, then $(\ref{PV})$ holds for every $\nu \in M(\R^d)$.
\end{thm}
Therefore, in order to prove Theorem \ref{GenThm}, it suffices to prove the existence of the limit in $(\ref{PV})$ in the case when $\nu=\mu$.

\section{Introductory lemmas}
In this section, we present some lemmas and remarks that will be used extensively throughout the paper.

\begin{remark}\label{intsym} Suppose $\nu$ is a measure, and $K$ an $s$-dimensional CZ-kernel.  Notice that if $x \in \mathcal{S}(\nu)$, then if $\varphi: \mathbb{R} \to \mathbb{R}$ is a Lipschitz function with $\int_0^\infty |\varphi'| \,dr < \infty$ and $\int |K(x-y)||x-y|^s |\varphi(|x-y|)| \,d\nu(y) < \infty$, then
     \begin{equation*} \begin{split}
      \int &K(x-y)|x-y|^s \varphi(|x-y|) \,d\nu(y) \\ & = -\int_0^\infty  \int_{B(x,r)} K(x-y)|x-y|^s \varphi'(r) \,d\nu(y) \,dr=0.
      \end{split}\end{equation*}
\end{remark}

\begin{remark}[Scaling]\label{scaling}  We will often rescale estimates to prove them for a unit scale. To this end we make some simple observations.  Fix $x\in \R^d$ and $r>0$.  For a measure $\nu$, set $\nu_{x,r}:=\frac{\nu(r\cdot+x)}{r^s}$. Observe that $\nu\in \Ms$ if and only if $\nu_{x,r}\in \Ms$.
\begin{enumerate}
\item Firstly, if $K$ is a CZ kernel, then the kernel
$$K_r=r^s K(r \,\cdot)$$ is a CZ kernel with the same constants (i.e. we can take $C_{K_r} = C_K$).
\item Secondly, a function $f\in \mathcal{F}_{x,r}$ if and only if the function $$f(r\cdot+x)\in \mathcal{F}_{0,1}.$$
\item Thirdly, for a collection $\mathbb{S}_{s}$ of symmetric measures associated to $K$ satisfying Assumptions \ref{SKdef}, the collection $$\widetilde{\mathbb{S}}_s=\{\nu_{x,r}:=\frac{\nu(r\cdot+x)}{r^s}: \mu\in \mathbb{S}_s\}$$ is a collection of symmetric measures associated to $K_r$ satisfying the Assumptions \ref{SKdef} (with the same constants $\Theta$ and $\SB$).
\item Combining the two previous observations we have
    $$\alpha_{\mu, \mathbb{S}_s}(B(x,r)) = \alpha_{\mu_{x,r}, \widetilde{\mathbb{S}}_s}(B(0,1)).
    $$
    \end{enumerate}
\end{remark}
\begin{lemma}\label{smoothK}Fix $\kap>0$.  Suppose that $\psi$ is a bounded Lipschitz continuous function satisfying $\psi\equiv 0$ on $B(0, \kap)$.  For every $x\in \R^d$, the function
$$F(y) =K(x-y)\psi(|x-y|)
$$
is a bounded Lipschitz continuous function, with \begin{equation}\label{fbounds}\|F\|_{\infty}\lesssim \frac{\|\psi\|_{\infty}}{\kap^s}\text{ and }\|F\|_{\Lip}\lesssim\frac{\|\psi\|_{\infty}}{\kap^{s+1}}+\frac{\|\psi\|_{\Lip}}{\kap^{s}}.\end{equation}
\end{lemma}

\begin{proof}  We first observe that, due to (\ref{size}) from Definition \ref{CZdef}, and the properties of $\psi$,
$$|F(y)|\lesssim \frac{1}{|x-y|^s}\|\psi\|_{\infty}\chi_{\R^d\backslash B(x,\kap)}(y)\lesssim \frac{\|\psi\|_{\infty}}{\kap^s}.$$

To see the Lipschitz property, fix $y,y'\in \R^d$.  If both $|x-y|<\kap$ and $|x-y'|<\kap$, then $|F(y)-F(y')|=0$, so assume that $|x-y|\geq \kap$.  If $|y-y'|\geq \tfrac{\kap}{3}$, then
$|F(y)-F(y')| \leq 2 \|F\|_{\infty}\lesssim \frac{\|\psi\|_{\infty}}{\kap^s}\lesssim \frac{\|\psi\|_{\Lip}}{\kap^{s+1}}|y-y'|.$

On the other hand, if $|y-y'|<\tfrac{\kap}{3}$, then $|x-y'|\geq \tfrac{2\kap}{3}$.  Therefore, we may apply (\ref{decay}) of Definition \ref{CZdef}, to get that $$|K(x-y)-K(x-y')|\lesssim \frac{|y-y'|}{|x-y'|^{s+1}}\lesssim \frac{|y-y'|}{\kap^{s+1}},$$
and consequently,
\begin{equation}\begin{split}\nonumber|F(y)-F(y')|  \leq &|K(x-y)-K(x-y')|\|\psi\|_{\infty}\\&+|K(x-y)[\psi(|x-y|)-\psi(|x-y'|)]|.
\end{split}\end{equation}
The first term is bounded by a multiple of $\frac{\|\psi\|_{\infty}}{\kap^{s+1}}|y-y'|$, while the second term is at most a constant multiple of $\frac{\|\psi\|_{\Lip}}{\kap^s}|y-y'|$.
\end{proof}

\subsection{\textbf{Thin boundary balls}}

\begin{lemma}\label{thinshell} Fix a measure $\mu$, $x\in \R^d$, $r>0$, and an even integer $M>0$. There exists an even integer $M'\in [M+2^{10},2^{11} M]$ such that
$$\mu(B(x, (M'+2^{10})r)\backslash B(x, (M'-2^{10})r))\leq \frac{2}{M}\mu(B(x,2^{11}Mr)).
$$
\end{lemma}

\begin{proof}
For $j=1,\dots, \tfrac{M}{2}$, set $A_j = B(x, (M+j2\cdot2^{10})r)\backslash B(x,(M+2(j-1)2^{10}r))$.  There are $M/2$ disjoint annuli $A_j$, all contained in the ball $B(x, (M+M2^{10})r)\subset B(x, 2^{11}Mr)$, so we must have that $\mu(A_j)\leq \frac{2}{M}\mu(B(x, 2^{11}Mr))$ for some $j$.  Set $M' = M+(2j-1)2^{10}\in [M+2^{10}, 2^{11}M]$.
\end{proof}

\subsection{The David-Mattila toolbox}

\begin{lemma}\label{comparison}
Let $\sigma$ be a complex measure on $\mathbb{R}^d$. Let $x,x' \in \mathbb{R}^d$ be such that $|x-x'| \leq \frac{1}{2} \dist(x, \supp \sigma)=: \rho$. Then
$$ |T (\sigma)(x) - T (\sigma)(x')| \lesssim \frac{|x-x'|}{\rho} \sup_{r>0} \frac{|\sigma|(B(x,r))}{r^n}.$$
\end{lemma}
\begin{proof}Using property (\ref{decay}) of Definition \ref{CZdef},
\begin{align*}
& |T(\sigma)(x)-T(\sigma)(x')|\leq  \int |K(x-y) - K(x'-y) |\,d |\sigma|(y) \\
& \lesssim \int\limits_{|x-y| > \rho} \frac{|x-x'|}{|x-y|^{n+1}}\,d|\sigma|(y) \lesssim  \frac{|x-x'|}{\rho} \sup_{r>0} \frac{|\sigma|(B(x,r))}{r^n},
\end{align*}
as required.
\end{proof}

\begin{lemma}[The David-Mattila lemma]\label{DMlem}
Fix a measure $\mu$, and $A>1$.  If
$$ D_{\mu}(B(x,A^{-(\ell+1)}r)) \leq \frac{1}{A} D_{\mu}(B(x,A^{-\ell}r)) \text{ for } \ell =0,...,L-2,$$

then we have
$$|T_r (\mu)(x) - T_{A^{-L}r}(\mu)(x)| \lesssim \frac{A^s}{A-1}D_\mu(B(x,r)).$$

\end{lemma}
\begin{proof}  Appealing to (\ref{size}) from Definition \ref{CZdef} we get that
\begin{align*}
|T_r(\mu)(x)-T_{A^{-L}r}(\mu)(x)| & = \left| \int_{A^{-L}r \leq |y-x| <  r} K(x-y) \,d\mu(y) \right| \\
& \lesssim \sum_{k=1}^{L} \int_{A^{-k} r \leq |y-x| < A^{-(k-1)}r} \frac{1}{|y-x|^s} \,d\mu(y) \\
& \lesssim \sum_{k=1}^{L} \frac{\mu(B(x,A^{-(k-1)}r))}{(A^{-k}r)^s}.
\end{align*}
But,
for $k=1,\dots, L-1$
$$ D_\mu(B(x,A^{-k}r)) \leq \frac{D_\mu(B(x,A^{-(k-1)}r))}{A} \leq \cdots \leq \frac{D_\mu(B(x,r))}{A^{k}},$$
from which we get that
$$\sum_{k=1}^{L}\frac{\mu(B(x,A^{-(k-1)}r))}{(A^{-(k-1)}r)^s} \lesssim \frac{1}{A-1}D_\mu(B(x,r)).$$
Whence,
$|T_r(\mu)(x) - T_{A^{-L}r}(\mu)(x)|\lesssim \frac{A^s}{A-1}D_{\mu}(B(x,r)).$
\end{proof}

\section{Basic estimates for measures with small transportation number} \label{estA}

For this section,  fix $\mu \in \Ms$. Fix $M\gg 1$ an even integer, $\alpha \ll 1$ and $\eps \ll 1$. Here \begin{itemize}
\item $M$ is an enlargement parameter which handles the non-local part of the integral operator,
\item $\alpha$ is the size bound for the transportation number, and
\item $\eps$ is a size threshold for the density of the measure.
\end{itemize}
We recall that a constant $C$ may change from line to line and may depend on $d$, $s$, the constants $C_K$ from Definition \ref{CZdef} and $\SB$ and  $\Theta$ from Assumptions \ref{SKdef}.

\subsection{Small Boundaries}\begin{lemma}[Small Annulus]\label{smallann}  Fix $\kap\in (0, 1/4]$.  There is a constant $C>0$ such that if $\alpha_{\mu}(B(x,r))< \alpha$ and $x_0 \in B(x,r/4)$, then
$$\frac{\mu(B(x_0,(1+\kap) r)\backslash B(x_0,r))}{r^s}\lesssim \frac{\alpha}{\kap}+\kap\Bigl.
$$
\end{lemma}

\begin{proof}  Due to Remark \ref{scaling}, we may suppose $x=0$ and $r=1$.  Fix $\kap\in [0, 1/4)$ and introduce a cutoff function $\psi_{\kap}$ with $\|\psi_{\kap}\|_{\Lip}\leq \frac{1}{\kap}$  which satisfies $\psi_{\kap}\equiv 1$ on $[1,1+\kap]$ and $\psi_{\kap}\equiv 0$ on $[0, 1-\kap]\cup [1+2\kap,\infty)$. We then denote by $\psi_{\kap,x_0}$ the bump function $\psi_{\kap,x_0}=\psi_\kap(|\cdot - x_0|)$.  Then $\kap\psi\in \F_{0,1}$ and so for a measure $\nu \in \SK$ with $\alpha_{\mu, \nu}(B(0,1))<\alpha$,
$$\mu(B(x_0,(1+\kap))\backslash B(x_0,1))\leq \int_{\R^d}\varphi \,\psi_{\kap,x_0}d\mu \leq c_{\mu, \nu}\int_{\R^d}\varphi\, \psi_{\kap,x_0}d\nu+\frac{\alpha}{\kap}.
$$
Since $x_0 \in B(0,\frac{1}{4})$ and  $\supp (\nu) \cap B(0,\frac{1}{8}) \neq \varnothing$, it follows that $B(x_0,\frac{1}{2}) \cap \supp(\nu) \neq \varnothing$. Hence, by property (\ref{SymSmallBdary}) from Assumptions \ref{SKdef}, and the fact that $\mu\in \Ms$, $$\frac{\I_{\mu}(B(0,1))}{\I_{\nu}(B(0,1))}\int_{\R^d}\varphi \,\psi_{\kap,x_0}d\nu \lesssim\kap \frac{\I_{\mu}(B(0,1))}{\I_{\nu}(B(0,1))}\I_{\nu}(B(0,1))\lesssim\kap\mu(B(0, 4)) \lesssim\kap,$$ as required.
\end{proof}

The following corollary enables us to move between rough and smooth cut-offs of a singular integral operator.

\begin{cor}[Rough to smooth cut-off]\label{rough2smooth}
Fix $\kap>0$, $x \in \R^d$ and $r>0$. If $\alpha_{\mu}(B(x,r))<\alpha$, then
$$|T([1 - \eta_{\kap, r,x}]\mu)(x)-T_r(\mu)(x)|\lesssim \frac{\alpha}{\kap} +\kap.
$$
\end{cor}

\begin{proof}  We may set $x=0$, $r=1$.  The left hand side of the inequality is
$$\Bigl|\int_{B(0, 1+\kap)\backslash B(0,1)}\eta_{\kap,1,0}(y)K(y)d\mu(y)\Bigl|.
$$
In this expression, $|K(y)|\lesssim 1$ on the domain of integration, so the expression is bounded by $C\mu(B(0, 1+\kap)\backslash B(0,1))$ from which we conclude the proof using Lemma \ref{smallann}.
\end{proof}

\subsection{Doubling properties}
\begin{lemma}[Doubling]\label{doubling} Fix $x\in \R^d$, $r>0$ and $x_0 \in B(x,\frac{r}{4})$. If $\alpha_\mu(B(x, r)) <\alpha$, then
$$D_\mu(B(x_0,2r)) \lesssim D_\mu(B(x_0,r)) +  \alpha.$$
\end{lemma}

\begin{proof}
Assume again that $x=0$ and $r=1$ (Remark \ref{scaling}).  Let $\phi\in \Lip_0(\R^d)$ be a cutoff function with $\|\phi\|_{\Lip} \leq 4$
 which satisfies  $\phi \equiv 1$ on $B(x_0,\frac{3}{4})$ and $\supp(\phi) \subset B(x_0,1)$.

For a measure $\nu\in \SK$, $B(0, \tfrac{1}{8})\cap\supp(\nu)\neq \varnothing$, so $B(x_0, \frac{3}{8}) \cap \supp(\nu) \neq \varnothing$.  Therefore, we may appeal to the small boundaries property (\ref{SymSmallBdary}) in Assumptions \ref{SKdef} repeatedly to obtain the following chain of inequalities
\begin{equation}\begin{split} \label{alpha}
\nu(B(x_0, 6)) \lesssim \nu(B(x_0, 3))\lesssim \nu(B(x_0, 3/2)\lesssim \nu(B(x_0,3/4)),
\end{split}\end{equation}
and, consequently, $\mathcal{I}_{\nu}(B(0,1))\lesssim \nu (B(x_0, 3/4))$.

Suppose now $\nu \in \SK$ is such that $\alpha_{\mu,\nu}(B(0,1) ) < \alpha$.  Then, since $\tfrac{1}{4}\phi \in \F_{0,1}$,\begin{equation}\begin{split}\nonumber\mu(B(x_0,2)) & \leq \I_\mu (B(0, 1)) \leq \frac{ \I_\nu (B(0, 1))}{ \int_{\mathbb{R}^d} \varphi \phi \,d\nu} \left(\int_{\mathbb{R}^d} \varphi \phi \,d\mu + 4 \alpha \right) \\& \lesssim  \mu(B(x_0,1)) + \alpha,
\end{split}\end{equation}
as required.
\end{proof}

We will often use the previous lemma in the form of the following Corollary.

\begin{cor}\label{longden}
There exists a constant $C_0>0$ (depending on $d$, $s$, and $\SB$) such that, if $Q \in [1,M^2]$, $x_0 \in B(x, \frac{r}{4})$, and $\alpha_\mu(B(x,\beta r)) < \alpha$ for all $\beta \in [1,M^2]$, then
\begin{enumerate}
\item if $D_\mu(B(x_0,Qr)) < \sqrt{\alpha}$, then $D_\mu(B(x_0,M^2r)) \lesssim M^{C_0} \sqrt{\alpha}$,
\item otherwise, $D_\mu(B(x_0,M^2r)) \lesssim M^{C_0} D_\mu(B(x_0,Qr))$.
\end{enumerate}
\end{cor}
\begin{proof}

From Lemma \ref{doubling} we have, for some $C>2$, that $D_{\mu}(B(x_0,2t)\leq C D_{\mu}(B(x_0,t))+C\alpha$ for all $t\in [r, M^2r]$.  Consequently
\begin{equation}\begin{split}\nonumber
D_\mu(B(x_0,M^2r)) &\leq C D_\mu(B(x_0,\frac{M^2}{2}r)) + C \alpha \\& \leq ... \leq C^{\log_2 M^2} D_\mu(B(x_0,Qr)) + \sum_{n=1}^{\log_2 M^2} C^n \alpha \\
&\lesssim M^{C_0} [D_\mu(B(x_0,Qr)) + \alpha],
\end{split}\end{equation}
for some constant $C_0>0$.  The alternative follows easily from this inequality.  Indeed, if $D_\mu(B(x_0,Qr)) < \sqrt{\alpha}$, it follows that
$$D_\mu(B(x_0,M^2r)) \lesssim M^{C_0} \sqrt{\alpha},$$
whereas if  $D_\mu (B(x_0,Qr)) \geq  \sqrt{\alpha}$, then,
$$ D_\mu(B(x_0,M^2r))  \lesssim  M^{C_0}D_{\mu}(B(x_0, Qr)),$$
as required.
\end{proof}

\subsection{The small density/reflection symmetry alternative}

\begin{lemma}[Alternative]\label{alt}   Provided that
$\alpha$ is small enough in terms of $M$ and $\varepsilon$, the following statement holds: If $\alpha_{\mu}(B(0,\beta r))<\alpha$ for every $\beta \in [\tfrac{1}{16}, 2^{50}\Theta]$, then for any measure $\nu\in \SK$ for which $\alpha_{\mu,\nu}(B(0, Mr))<\alpha$, either
\begin{enumerate}
\item there exists $\wt x \in B(0,r)$ with $\nu$ reflection symmetric in $B(\wt x,32 r)$, or there exists $\wt x \in B(0,(32 \Theta + 1)r)$ such that $\nu$ is reflection symmetric in $B(\wt x, 64\cdot 32 \Theta r)$, or
\item $\mu(B(0, 2^{50}\Theta r))\lesssim \eps r^s $.
\end{enumerate}
\end{lemma}

\begin{proof} Set $r=1$.  Fix $\eps'>0$.  Suppose first that $\mu(B(0, \frac{1}{16}))\geq \eps'$.  Fix $f$ satisfying $f \equiv 1$ on $B(0,1/16)$, $f\geq 0$, $\supp(f)\subset B(0, 1/8)$ and $\|f\|_{Lip}\leq C.$ Then
$$\Bigl| \frac{1}{M^s}\int_{\R^d} f \,d\Bigl(\mu-\frac{\I_{\mu}(B(0,M))}{\I_{\nu}(B(0,M))}\nu\Bigl) \Bigl| \leq  C\alpha_{\mu,\nu}(B(0,Mr))\leq C\alpha.$$
Since $\int_{\R^d}fd\mu>\eps'$, rearranging this gives
$$\Bigl|\int_{\R^d}fd\nu\Bigl|\geq \bigl(\eps'-C\alpha M^s\bigl) \frac{\I_{\mu}(B(0,M))}{\I_{\nu}(B(0,M))}>0
$$
provided that $CM^s\alpha<\eps'$.  Under this condition on $\alpha$ we therefore can find an element $\wt x_\nu \in \supp(\nu)\cap B(0, 1/8)$.  Applying (\ref{SymReflSym}) in the definition of $\nu \in \SK$ with scale 32 at $\wt x_\nu$, we find that either $\nu$ is reflection symmetric in  about $\wt x_\nu$ in $B(\wt x_\nu,32) $, or there exists $\wt x'_\nu \in B(\wt x_\nu,32\Theta)\subset B(0,32 \Theta +1)$ with $\nu$ reflection symmetric about $\wt x'_\nu$ in $B(\wt x'_\nu, 64\cdot 32 \Theta)$.

Now suppose instead $\mu(B(0, 1/16))\leq \eps'$ holds. Since $\alpha_{\mu}(B(0, Q))<\alpha$ for every $Q\in [\tfrac{1}{16}, 2^{50}\Theta]$, we may repeatedly apply Lemma \ref{doubling} and get that $D_{\mu}(B(0, 2^{50}\Theta))\lesssim D_{\mu}(B(0, 1/16)) +\alpha \lesssim\eps'+\alpha.$ Therefore, if $\eps'=c\eps$ for a sufficiently small constant $c>0$ (depending on $d$, $s$ and $\Theta$), and $\alpha$ is small enough in terms of $\eps$, we get that (2) of the alternative holds.
\end{proof}

\section{The three main estimates}\label{threeestimates}

In this section we derive our main technical estimates.  Again fix $\mu\in \Ms$, $M\gg 1$ an even integer, $\alpha
\ll 1$ and $\eps \ll 1$.   We use the notation $\dashint_E fd\mu = \frac{1}{\mu(E)}\int_E fd\mu$. Recall the bump function $\eta_{\kap, r,x}$ from Section \ref{etadef}.

\subsection{Long range comparison lemma}

\begin{prop}\label{proplong}  If $x\in \R^d$, $r>0$, and $M'\in [M+2^{10}, 2^{11}M]$ is chosen as in Lemma \ref{thinshell}, then
\begin{equation}\begin{split}\nonumber\Bigl|T([1\!-\!\eta_{\frac{1}{M'},M'r,x}]\mu)(x) - \dashint_{B(x_0,Qr)}\!\!\!\!&T([1-\eta_{\frac{1}{M'},M'r,y}]\mu)(y)d\mu(y)\Bigl|\lesssim\frac{1}{M}
\end{split}\end{equation}
for every $Q\in [1,2^8]$ and $x_0\in B(x,  r)$.
\end{prop}
\begin{proof}  Without loss of generality, take $r=1$.  For three points $y,y',y''\in B(x_0, 2^8)$, we observe two estimates.  Firstly,
\begin{equation}\begin{split}\nonumber\Bigl|\int_{\R^d}K(y-z)&\bigl[\eta_{\frac{1}{M'},M',y'}(z)-\eta_{\frac{1}{M'}, M', y''}(z)]d\mu(z)\Bigl|\\&\leq \int_{B(x_0, (M'+2^{10})\backslash (M'-2^{10}))}|K(y-z)|d\mu(z)\\&
\lesssim \frac{1}{M^s}\mu (B(x_0,(M'+2^{10})\backslash B(M'-2^{10})))\\&\stackrel{\text{Lemma }\ref{thinshell}}{\lesssim} \frac{1}{M}D_{\mu}(B(x_0, 2^{12}M))\lesssim \frac{1}{M}.
\end{split}\end{equation}
  Secondly, arguing as in Lemma \ref{comparison}, we have the estimate
$$\int_{\R^d}|K(y-z)-K(y'-z)|[1-\eta_{\frac{1}{M'}, M', y''}(z)]d\mu(z)\lesssim \frac{1}{M'}\lesssim \frac{1}{M}.
$$
The lemma now follows from the triangle inequality.
\end{proof}

\begin{cor} \label{corlong}   If, in addition to the assumptions of Proposition \ref{proplong}, $\alpha_{\mu}(B(x, \beta r))\leq \alpha$ for all $\beta \in [\frac{1}{M}, M^2]$, then for any $\delta>0$, we can
\begin{itemize}
\item choose $M$ sufficiently large in terms of $\delta$, and then
\item choose  $\alpha$ sufficiently small in terms of $M$ and $\delta$,
\end{itemize}
so that for all $t\in[\tfrac{r}{M},Mr]$,
$$\Bigl|T_t(\mu)(x) - \dashint_{B(x_0,Qr)}\!\!\!\!T(1-\eta_{\frac{1}{M'},M'r,y}\mu)(y)|d\mu(y)\Bigl|\leq \delta.$$
\end{cor}
\begin{proof}
Suppose $r=1$.  Fix $\kap\in (0, \tfrac{1}{2})$.   In order to apply Proposition \ref{proplong}, we should estimate $|T_{t}(\mu)(x)- T([1\!-\!\eta_{\frac{1}{M'},M',x}]\mu)(x)|$, which is at most
\begin{equation}\begin{split}\nonumber
\Bigl| T_t(\mu)(x)- &T([1 - \eta_{\kap, t,x}]\mu)(x)\Bigl|+ \Bigl|T([\eta_{\frac{1}{M'},M',x} - \eta_{\kap, t,x}]\mu)(x)\Bigl|= I+II.
\end{split}\end{equation}
Appealing to Corollary \ref{rough2smooth} we observe that $I \lesssim \bigl(\frac{\alpha}{\kap} +\kap\bigl)$.  To estimate II, first observe that $\|\eta_{\kap, t,x}\|_{\Lip}\lesssim \frac{M}{\kap}$ ($t\geq \frac{1}{M}$), and $\|\eta_{\frac{1}{M'},M',x}\|_{\Lip}\lesssim 1$.   Therefore, from Lemma \ref{smoothK} we infer that the function $ f(y) = K(x-y)[\eta_{\frac{1}{M'}, M',x}(y)-\eta_{\kap, t,x}(y)]$ has Lipschitz norm at most $\frac{CM}{\kap^{s+1}}$, and so the function $\tfrac{c\kap^{s+1}}{M^2} f\in \F_{x,M'}$ for some small constant $c>0$.  On the other hand, if $\nu\in \S$ and $x\in \S(\nu)$, then from Remark \ref{intsym} we have $\int_{\R^d}f(y)d\nu=0$.  Insofar as $\alpha_{\mu}(B(x, M'))<\alpha$, we therefore obtain
$$II=\Bigl|\int_{\R^d}fd\mu\Bigl|\lesssim M^{s+2}\frac{\alpha}{\kap^{s+1}}.$$
 The quantity we want to estimate is therefore bounded by a constant multiple of $\frac{1}{M}+\kap + \frac{\alpha}{\kap}+ \frac{\alpha M^{s+2}}{\kap^{s+1}}$. Setting $\kap=\alpha^{\frac{1}{2(s+1)}}$, we see that this quantity can be made smaller than $\delta$ by first fixing $M$ small enough in terms of $\delta$ and then $\alpha$ small enough in terms of $\delta$ and $M$.
\end{proof}
\subsection{Intermediate contribution lemma}

\begin{prop}\label{propinter}  If $x_0\in B(x, r)$, and $\alpha_{\mu}(B(x, \beta r))< \alpha$ for every $\beta \in [1,M^3]$, then for any $\delta>0$, we can choose $\alpha$ small enough in terms of $M$ and $\delta$ such that
$$\Bigl|\int_{B(x_0,Qr)}T([\eta_{\frac{1}{M},Mr,y}- \eta_{\frac{1}{Q'}, Q'r,y}]\mu)(y)d\mu(y)\Bigl|\leq C \delta \,\mu(B(x_0,Qr))  $$
for any $M > Q'\geq Q$ and $Q\in [4 , M/2]$.
\end{prop}
\begin{proof}
Again without loss of generality we assume that $r=1$.  The hypotheses of Corollary \ref{longden} are satisfied, and so we have that either
\begin{enumerate}
\item $\mu(B(x_0,M^2)) \lesssim M^{C_0} \sqrt{\alpha}$, or
\item  $\mu(B(x_0,M^2)) \lesssim M^{C_0} \mu(B(x_0,Q))$, for some constant $C_0$.
\end{enumerate}

If (1) occurs we can appeal to the size estimate (\ref{size}) from Definition \ref{CZdef}, obtaining
\begin{equation}\begin{split}\nonumber\Bigl|\int_{B(x_0,Q)}&T([\eta_{\frac{1}{M},M,y}- \eta_{\frac{1}{Q'}, Q',y}]\mu)(y)d\mu(y)\Bigl| \\
&= \Bigl|\int_{B(x_0,Q)}\int_{\mathbb{R}^d} K(y-z)[\eta_{\frac{1}{M},M,y}(z)- \eta_{\frac{1}{Q'}, Q',y}(z)]\,d\mu(z) \,d\mu(y) \Bigl|\\
&\leq \int_{B(x_0,Q)}\int_{B(y,2M) \backslash B(y,Q')} \Bigl|K(y-z)\Bigl|\,d\mu(z) \,d\mu(y) \\
& \lesssim \mu(B(x_0,Q)) \frac{\mu(B(x_0,M^2))}{(Q')^s} \lesssim \mu(B(x_0,Q)) M^{C_0} \sqrt{\alpha}\\
&<\delta\mu(B(x_0,Q))
\end{split}\end{equation}
provided $\alpha$ is sufficiently small in terms of $M$ and $\delta$.

Suppose now that we are in the scenario (2). We wish to estimate
\begin{equation}\label{someintegral} \Bigl| \int_{B(x_0,Q)}T([\eta_{\frac{1}{M},M,y}- \eta_{\frac{1}{Q'}, Q',y}]\mu)(y)d\mu(y)\Bigl|.\end{equation}
 Instead we will consider
  \begin{equation}\begin{split}\label{instead}
 \Bigl| \int_{\mathbb{R}^d}T([\eta_{\frac{1}{M},M,y}- \eta_{\frac{1}{Q'}, Q',y}]\mu)(y)\eta_{\sqrt{\alpha},Q,x_0}(y)\,d\mu(y)\Bigl|,
  \end{split}\end{equation}
  because the difference between (\ref{someintegral}) and (\ref{instead}) can be bounded using the size estimate (\ref{size}) from Definition \ref{CZdef} by
 \begin{equation}\begin{split}\nonumber
   \int_{\R^d}  &\left( \eta_{\sqrt{\alpha},Q,x_0}(y)-\chi_{B(x_0,Q)}(y)\right) |T([\eta_{\frac{1}{M},M,y}- \eta_{\frac{1}{Q'}, Q',y}]\mu)(y)| \,d\mu(y)\\
& \lesssim  \frac{\mu(B(x_0,M^2))}{ (Q')^s}  \int_{\R^d} \left( \eta_{\sqrt{\alpha},Q,x_0}(y)-\chi_{B(x_0,Q)}(y)\right)\,d\mu(y) \\
 &\stackrel{\text{Lemma }\ref{smallann}}{\lesssim} \frac{\mu(B(x_0,M^2))}{(Q')^s} \sqrt{\alpha}Q^s\lesssim \sqrt{\alpha}M^{C_0}\mu(B(x_0,Q)<\frac{\delta}{3}\mu(B(x_0,Q)),
 \end{split}\end{equation}
 as long as $\alpha$ is small enough in terms of $M$ and $\delta$.

 To estimate (\ref{instead}) we introduce $\nu \in \SK$ such that $\alpha_{\mu,\nu}(B(x,M^2/8)) \leq \alpha$.  Our goal will be to replace $\nu$ by $\mu$ in every instance in (\ref{instead}), since the resulting integral is zero due to the fact that $\nu$ is symmetric (as we shall show below).

 We first replace $\mu$ by $\nu$ in the operator appearing in (\ref{instead}), so we need to estimate
 \begin{equation}\begin{split}\label{int}&\Bigl|\int_{\R^d} \eta_{\sqrt{\alpha},Q,x_0}(y) T([\eta_{\frac{1}{M},M,y}- \eta_{\frac{1}{Q'}, Q',y}](\mu - c_{\mu,\nu}\nu))(y)\,d\mu(y)\Bigl|,
 \end{split}\end{equation}
 where  $c_{\mu,\nu}=\frac{I_{\mu}(B(x,M^2/8))}{I_{\nu}(B(x,M^2/8))}$. (The factor of $1/8$ here is merely because then $\mathcal{I}_{\mu}(B(x, M^2/8))\leq \mu(B(x, M^2/2))\leq \mu(B(x_0, M^2))$.)

To estimate (\ref{int}), we observe from Lemma \ref{smoothK} that, for any $y \in B(x_0,2Q)$ (which contains $\supp \eta_{\sqrt{\alpha},Q,x_0}$) the function $$z \to K(y-z)(\eta_{1/M,M,y}(z)-\eta_{1/Q',Q'y}(z))$$
has Lipschitz norm at most $C$ (recall $Q'\geq 1$). Consequently,
\begin{equation} \begin{split} \sup_{y \in \supp (\eta_{\sqrt{\alpha},Q,x_0})} \Bigl| &T([\eta_{\frac{1}{M},M,y}- \eta_{\frac{1}{Q'}, Q',y}](\mu - c_{\mu,\nu}\nu))(z)\Bigl| \\
& \lesssim M^{2(s+1)} \alpha_{\mu}(B(x,M^2/8)) \lesssim  M^{2(s+1)} \alpha.
\end{split} \end{equation}

 Therefore we can bound  (\ref{int}) by a constant multiple of
 $$M^{2(s+1)} \alpha \mu(B(x_0,2Q))\lesssim M^{C_0+2(s+1)}\alpha \mu(B(x_0,Q))<\frac{\delta}{3}\mu(B(x_0, Q)),$$
provided $\alpha$ is sufficiently small in terms of $M$ and $\delta$.

Our next step is to estimate
 \begin{equation}\label{int2}
\Bigl|\int_{\mathbb{R}^d} \eta_{\sqrt{\alpha},Q,x_0}(y)T([\eta_{\frac{1}{M},M,y}- \eta_{\frac{1}{Q'}, Q',y}](c_{\mu,\nu}\nu))(y)\,d(\mu-c_{\mu,\nu} \nu)(y)\Bigl|.
\end{equation}
We next claim that, on $B(x_0, 2Q)$, the function
$$y \to T([\eta_{\frac{1}{M},M,y}- \eta_{\frac{1}{Q'}, Q',y}]\,\nu)(y),$$
is $\nu(B(x_0,3M))$-Lipschitz, and also bounded by $\nu(B(x_0, 3M))$.

For $y\in B(x_0, Q)$, the function is an integral (with respect to $\nu$) over $z\in B(x_0, 3M)$ of $F_z(y) = K(y-z)[\eta_{\frac{1}{M},M,y}(z)- \eta_{\frac{1}{Q'}, Q',y}(z)]$.  But for every $z\in \R^d$, Lemma \ref{smoothK} ensures that $F_z$ has Lipschitz norm and $L^{\infty}$ norm at most some constant $C>0$, so the claim follows.

Consequently, recalling the definition of $c_{\mu.\nu}$ we use the fact that $\alpha_{\mu,\nu}(B(x,M^2/8)) \leq \alpha$ to estimate (\ref{int2}) by
\begin{equation}\begin{split}\nonumber C\frac{\mathcal{I}_{\mu}(B(x_0, M^2/8))}{\mathcal{I}_{\nu}(B(x_0, M^2/8))}\frac{\nu(B(x_0, 3M))}{\sqrt{\alpha}}M^{2(s+1)}\alpha&\lesssim M^{C_0+2(s+1)}\sqrt{\alpha}\mu(B(x_0,Q))\\&<\frac{\delta}{3}\mu(B(x_0,Q)),
\end{split}\end{equation}
provided $\alpha$ is small enough in terms of $\delta$ and $M$.

Our last claim is that  $$ \int_{\R^d} \eta_{\sqrt{\alpha},Q,x_0}(y)\int_{\R^d} K(y-z)[\eta_{\frac{1}{M},M,y}(z)- \eta_{\frac{1}{Q'}, Q',y}(z)] \,d \nu(z)   \,d\nu(y)=0. $$
In order to see this, we focus once more in the inner integral. For $y \in B(x_0,2Q) \cap \supp(\nu)$, the fact that $\nu$ is symmetric, we get that from Remark \ref{intsym},
$$\int_{\R^d} K(y-z)[\eta_{\frac{1}{M},M,y}(z)- \eta_{\frac{1}{Q'}, Q',y}(z)] \,d \nu(z)=0 \text{ for every }y \in \supp(\nu).$$
Combining these estimates the proposition follows.
\end{proof}

\subsection{Unit scale averaging lemma}
\begin{prop}\label{propunit} Assume that $x_0 \in B(x,r)$, $\alpha_{\mu}(B(x, \beta r)) < \alpha$ for every $\beta \in [1,M^2]$ and let $\gamma\in (0,1)$.

Provided that $\gamma M \gg1$, then for any $\delta>0$,
\begin{itemize}
\item we can choose $\eps$ sufficiently small depending on $\delta$, and
\item $\alpha$ sufficiently small depending on $\eps$ and $M$
\end{itemize}
 such that, if $\alpha_{\mu,\nu}(B(x,\gamma M r))\leq \alpha$ for $\nu\in \SK$ satisfying either
\begin{enumerate}
\item $\nu$ is reflection symmetric in $B(x_0, 32 r)$, or
\item $\mu(B(x, 32 r))\leq \eps r^s$, in which case we set $x_0=x$,
\end{enumerate}
then
\begin{equation}\label{unitint}\Bigl|\int_{4}^{8}\int_{B(x_0, Qr )}  T(\eta_{\frac{1}{16},16 r,y}\chi_{B(x_0,Qr)^c}\mu)d\mu(y)dQ\Bigl|< \delta \mu(B(x_0,2^{10}r)).
\end{equation}
\end{prop}
\begin{proof}We again assume $x_0=0$ and $r=1$ (Remark \ref{scaling}).
Rewrite the integral appearing in (\ref{unitint}) as
\begin{align*}
& \int_{4}^{8}\int_{B(0, Q)}  T(\eta_{\frac{1}{16},16,y}\;\chi_{B(0,Q)^c}\mu)d\mu(y)dQ\\
&= \int_{4}^{8} \int_{B(0,Q)} \int_{B(0,Q)^c} K(y-z)\eta_{\frac{1}{16},16,y}(z) \,d\mu(z) \,d\mu(y) \,dQ \\
&=\int_{\mathbb{R}^d} \eta_{1,8,0}(y)\int_{\mathbb{R}^d} \eta_{\frac{1}{16},16,y}(z)  K(y-z) \int_{4}^{8}\chi_{\left\{\substack{ |y|< Q,\\ |z|>Q}\right\}}(Q) \,dQ \,d\mu(z)\,d\mu(y).
\end{align*}
Denote by $L(y,z)$ the function
$$\int_{4}^{8}\chi_{\left\{\substack{ |y|< Q,\\ |z|>Q}\right\}}(Q) \,dQ = \left[ \{ \min(8,|z|) -4 \}_+ - \{ \max(4,|y|) - 4\}_+ \right]_+.$$
   Notice that $|L(y,z)| \lesssim |y-z|$, and for any $z$, $L(\cdot,z)$ is a Lipschitz function with Lipschitz norm $\|L(\cdot,z)\|_{\lip} \lesssim 1$. From now on, we denote $H(y,z)=K(y-z)L(y,z)$ and
   $$F_{\mu}(y)=\int_{\R^d} H(y,z) \eta_{\frac{1}{16},16,y}(z) \,d\mu(z).$$
   Fix $\kap >0$.  Decompose $F_\mu$ into its local and its non-local parts,
      $$F_{\mu}(y)=F_{\mu}^\kap(y)+F_{\mu,\kap}(y),$$
   where
   $$F_{\mu,\kap}(y)=\int_{\R^d} H(y,z)(\eta_{\frac{1}{16},16,y}(z)-\eta_{\frac{1}{16},\kap,y}(z))\,d\mu(z),$$
   and
   $$F_{\mu}^\kap(y)=\int_{\R^d} H(y,z)\eta_{\frac{1}{16},\kap,y}(z) \,d\mu(z).$$
   Appealing to the size estimates of $K$ and $L$, we have $|H(y,z)|\lesssim \frac{1}{|y-z|^{s-1}}$.  Therefore
    $$|F_{\mu}^\kap(y)|\leq  \int_{|y-z|\leq 2\kap} \frac{1}{|y-z|^{s-1}} \,d\mu(z) \lesssim \kap,$$
  and
   $$\Bigl|  \int_{\mathbb{R}^d} F_{\mu}(y)\eta_{1,8,0}(y) \,d\mu(y) \Bigl| \leq \Bigl| \int_{\mathbb{R}^d} F_{\mu,\kap}(y) \eta_{1,8,0}(y) \,d\mu(y) \Bigl|+ \frac{\delta}{3}\mu(B(0,16))$$
if $\kap$ is a chosen to be a small constant multiple of $\delta$.

    Fix $\nu \in \SK$ with $\alpha_{\mu,\nu}(B(x,\gamma M))\leq \alpha$ as in the statement of the proposition.   In the case $(2)$, we have that for $y\in B(0,9)$, the set $B(y,20)\subset B(0,32)$ and as $\mu(B(0, 32))\leq \eps$ we get that
   $$|F_{\mu,\kap}(y)|\leq \int_{\kap  \leq |y-z|\leq 20  } \frac{1}{|y-z|^{s-1}} \,d\mu(z) \leq \kap^{1-s} \varepsilon\lesssim \delta^{1-s}\eps.$$
   Integrating this bound with respect to the measure $\eta_{1,8,0}(y)d\mu(y)$, we therefore infer that
   $$   \Bigl|  \int_{\R^d} F_{\mu}(y)\eta_{1,8,0}(y) \,d\mu(y) \Bigl| <\delta \mu(B(0,16)),$$
as long as $\eps$ is sufficiently small in terms of $\delta$.

 We now consider the case when $\nu$ satisfies $(1)$, i.e. $\nu$ is reflection symmetric in $B(0,32)$, with $0 \in B(x,1)$. Our goal will be to replace
 \begin{equation}\label{doublemuunit} \int_{\R^d} F_{\mu,\kap}(y)\eta_{1,8,0}(y) \,d\mu(y) \end{equation}
 with
 $$  \int_{\R^d} F_{\nu,\kap}(y)\eta_{1,8,0}(y) \,d\nu(y),$$
 and show that the latter integral equals zero.
 Of course, we may assume that $\mu(B(x,32)) \geq \varepsilon$.  Hence, as long as $\sqrt{\alpha} \leq \varepsilon$, we can apply Corollary \ref{longden} in order to get
\begin{equation}\label{dens} D_\mu(B(x,M^2))  \lesssim M^{C_0}D_{\mu}(B(x, 32)) .\end{equation}
We next claim that
$$\|F_{\mu,\kap}\|_{\Lip(B(0,9))} \lesssim \frac{\mu(B(0,32))}{\kap^{s+1}}.$$
To see this, we first observe from Lemma \ref{smoothK} that, for any $z\in B(0,32)$, the function $y\mapsto W(y,z) = K(y-z)L(y,z)(\eta_{\frac{1}{16},16,y}(z)-\eta_{\frac{1}{16},\kap,y}(z))$ has Lipschitz norm at most $\frac{C}{\kap^{s+1}}$.  On the other hand, if $y\in B(0,9)$, then $z\mapsto W(y,z)$ is supported in $B(0, 32)$.  Since $F_{\mu,\kap}(y) = \int_{R^d}W(y,z)d\mu(z)$, the claim follows by combining these two observations.

   We now begin estimating (\ref{doublemuunit}).  Since $\frac{C \kap^{s+1}}{\gamma M}F_{\mu,\kap}\in \F_{0,\gamma M}$,
   \begin{align*}
    \Bigl|\int_{\mathbb{R}^d} F_{\mu,\kap}(y)\eta_{1,8,0}(y) \,d\left(\mu- \frac{I_\mu(B(x,\gamma M))}{I_\nu(B(x,\gamma M))}\nu \right)(y)\Bigl|& \lesssim \frac{(\gamma M)^{s+1}}{ \kap^{s+1}} \alpha \mu(B(0,32)) \\
   & \lesssim \frac{M^{s+1}}{ \kap^{s+1}} \alpha \mu(B(0,32)),   \end{align*}
   and the right hand side is at most $\frac{\delta}{3}\mu(B(0, 32))$ provided that $\alpha$ is sufficiently small in terms on $M$ and $\kap$.

  We next consider the expression
  \begin{equation}\label{int3}\frac{I_\mu(B(x,\gamma M))}{I_\nu(B(x,\gamma M))} \int_{\R^d}  \eta_{1,8,0}(y) \left(F_{\mu,\kap}(y)-\frac{I_\mu(B(x,\gamma M))}{I_\nu(B(x,\gamma M))}F_{\nu,\kap}(y) \right) \,d \nu(y).\end{equation}

  Write $F_{\mu,\kap}(y)-\frac{I_\mu(B(x,\gamma M))}{I_\nu(B(x,\gamma M))}F_{\nu,\kap}(y)$ as
  $$ \int_{\R^d} H(y,z)(\eta_{\frac{1}{16},16,y}(z)-\eta_{\frac{1}{16},\kap,y}(z)) \,d\left(\mu- \frac{I_\mu(B(x,\gamma M))}{I_\nu(B(x,\gamma M))}\nu\right)(z).$$
 Lemmma \ref{smoothK} yields that the function
   $$ z \to H(y,z)(\eta_{\frac{1}{16},16,y}(z)-\eta_{\frac{1}{16},\kap,y}(z))$$
   is $ \frac{C}{\kap^{s+1}}$-Lipschitz, and supported in $B(y,17)$. So, since $\alpha_{\mu,\nu}(B(x,\gamma M)) < \alpha$ yields that
      \begin{align*}
  |F_{\mu,\kap}(y)-\frac{I_\mu(B(x,\gamma M))}{I_\nu(B(x,\gamma M))}F_{\nu,\kap}(y)|\lesssim  \frac{M^{s+1}}{\kap^{s+1}} \alpha.
    \end{align*}
    But now (\ref{dens}) ensures that $\frac{I_\mu(B(x,\gamma M))}{I_\nu(B(x,\gamma M))}\nu(B(0,9))\lesssim M^{C_0}\mu(B(x, 32))$, so we estimate (\ref{int3}) by a constant multiple of $\frac{M^{s+1}}{\kap^{s+1}}M^{C_0} \alpha \mu(B(x,32)) <\frac{\delta}{3}\mu(B(0, 64))$ as long as $\alpha$ is sufficiently small in terms of $\kap$ and $M$.

Finally, the reflection symmetry of $\nu$ comes in to play. Notice that $H(y,z)$ satisfies that $H(-y,-z)=-H(y,z)$ and that
$$\eta_{1,8,0}(y)(\eta_{\frac{1}{16},16,y}(z)-\eta_{\frac{1}{16},\kap,y}(z)) =\eta_{1,8,0}(-y)(\eta_{\frac{1}{16},16,-y}(-z)-\eta_{\frac{1}{16},\kap,-y}(-z))  .$$  Combined with the fact that $\nu$ is reflection symmetric about $0$ in $B(0,32)$, this  implies that
 \begin{align*}
   \int_{\R^d}  \int_{\R^d}\eta_{1,8,0}(y) [\eta_{\frac{1}{16},\kap,y}(z)-\eta_{\frac{1}{16},16,y}(z)] H(y,z) \,d \nu (z)\,d \nu(y)=0.
  \end{align*}
  Bringing our estimates together, we conclude that, in case (b) we arrive at
  $$ \Bigl|  \int_{\R^d} F_{\mu}(y)\eta_{1,8,0}(y) \,d\mu(y) \Bigl| \leq \delta \mu(B(0,2^{10})),$$
  and the lemma is proved.
  \end{proof}

\section{Gluing it all together:  The Proof of Theorem \ref{GenThm}}

In this section we complete the proof of Theorem \ref{GenThm}.  Fix $\mu\in \Ms$, and the CZO $T$ associated to a CZ-kernel $K$ which is bounded in $L^2(\mu)$.  First observe that $T$ is also bounded in $L^2(\mu|E)$, where $\mu(E)<\infty$.  Consequently, we may (and will) assume that $\mu$ is a finite measure.

From Section \ref{reduction} we recall that it suffices to show that the limit (\ref{PV}) holds with $\nu=\mu$.  To this end, we recall a (particular case of a) theorem due to Mattila and Verdera, which states the existence of the weak limit for certain CZOs. For completeness, the short proof is presented in Appendix A.
\begin{thm}(Mattila-Verdera)\label{Weak}
Fix $s\in (0,d)$.  Let $\mu \in \Ms$ be a finite measure.  If the CZO $T$ associated to $K$ is bounded in $L^2(\mu)$, $T_{r}(\mu) \in L^2(\mu)$ converges weakly in $L^2(\mu)$ to a function $\wt{T}(\mu)$ as $r \rightarrow 0$.  Moreover
$$\widetilde{T}(\mu)(x)= \lim_{\substack{|B_r| \rightarrow 0\\ x \in \frac{1}{2}B_r}} \frac{1}{\mu(B_r)} \int_{B_r} T(\chi_{B^c_r}\,\mu)\,d\mu \text{    for  }\mu\text{-a.e. } x \in \mathbb{R}^d,$$
where $B_r$ denotes a ball of radius $r$.
\end{thm}

To prove Theorem \ref{GenThm}, we shall show that  $$\lim_{r\to 0}T_r(\mu)(x)=\wt T(\mu)(x)$$
for every $x\in \R^d$ for which
\begin{itemize}
\item $\lim_{r \to 0} \alpha_{\mu}(B(x,r))=0$,
\item $\wt T(\mu)(x)$ is well defined, and moreover
$$\lim_{r\to 0}\sup\limits_{B_r: \, \,x\in \tfrac{1}{2}B_r}\Bigl|\frac{1}{\mu(B_r)}\int_{B_r} T(\chi_{\R^d\backslash B_r }\mu)(y)d\mu(y)-\wt T(\mu)(x)\Bigl|=0,
$$

\end{itemize}

Fix $\delta\in (0,1)$.  Introduce $\eps\ll \delta$, $M\gg \frac{1}{\delta}$ an even integer and $\alpha \ll \delta$, with $M$ sufficiently large in terms of $\delta$, $\eps$ sufficiently small in terms of $\delta$, and $\alpha$ sufficiently small in terms of $\eps$, $M$ and $\delta$, so that Corollary \ref{corlong}, Proposition \ref{propinter}, and Proposition \ref{propunit} can be applied with this choice of $\delta$.

Fix $t_0>0$ such that \begin{equation}\label{epsdiff}\Bigl|\frac{1}{\mu(B_r)}\int_{B_r} T(\chi_{B^c_r} \,\mu)(y)d\mu(y)-\wt T(\mu)(x)\Bigl|\leq \delta\end{equation} whenever $r<M^3 t_0$ and $ x\in \frac{1}{2}B_r $, and also $$\alpha_{\mu}(B(x, r))< \alpha\text{ whenever }r\leq M^3t_0.$$

For $r<t_0$, we want to compare $T_r(\mu)(x)$ to a collection of averages of the form $\frac{1}{\mu(B_{r'})}\int_{B_{r'}} T(\chi_{B^c_r}\, \mu)(y)d\mu(y)$ with $x\in \frac{1}{2}B_{r'}$ and $r'$ comparable to $r$.  The averaging process here is with a view to applying Proposition \ref{propunit}.  To formally carry this out requires an initial reduction to doubling scales.

\subsection{Reduction to Doubling scales} We consider two cases.\\

\noindent \textbf{Case 1.} Suppose first that $D_{\mu}(B(x,r))<\eps$.  Fix $A=2^{30}\Theta $.  Suppose that, for $\ell = 0,\dots,L-1$ we have
\begin{equation}\label{nondoub}D_{\mu}(B(x, \frac{r}{A^{\ell}}))> AD_{\mu}(B(x, \frac{r}{A^{\ell+1}})).
\end{equation}
Then for $\ell=0, \dots, L$,
\begin{equation}\label{nondoubdecay}\mu(B(x, A^{-\ell}r))\leq A^{-\ell(s+1)}\mu(B(x,r))\leq A^{-\ell(s+1)}\eps r^s,
\end{equation}
and consequently, by the David-Mattila Lemma (Lemma \ref{DMlem}),
\begin{equation}\label{nondoubnoprob}\int_{B(x,r)\backslash B(x, A^{-{L+1}}r)}|K(x,r)|d\mu(y)\lesssim \eps.
\end{equation}
If it happens that (\ref{nondoubdecay}) holds for every $\ell\in \mathbb{N}$, then we have that the integral $\int_{\R}K(x-y)d\mu(y)$ converges absolutely, and so certainly the principal value exists at this $x$.  Hence we may assume that there exists $L\in \mathbb{N}$ such that (\ref{nondoub}) holds for all $\ell<L$ and fails for $\ell=L$.  Set $r_0 = r/A^{L+1}$.  Then we say that $r_0$ is a doubling scale,
\begin{equation}\label{doublr0}D_{\mu}(B(x, Ar_0))\leq AD_{\mu}(B(x,r_0)),
\end{equation}
and (\ref{nondoubnoprob}) holds.

\textbf{Case 2.}  If $D_{\mu}(B(x,r))>\eps$, then set $r_0=r$.  Since $\alpha \ll \eps$, and certainly $\alpha_{\mu}(B(x,t))<\alpha$ for all $t\in [r_0, Ar_0]$, applying Lemma \ref{doubling} $C\log A$ times ensures that (\ref{doublr0}) holds.\\

Notice that in either case, $r_0\leq r$, and from (\ref{nondoubnoprob})
\begin{equation}\label{nondoubnoprob2} |T_{r}(\mu)(x) - T_{ r_0}(\mu)(x)|\lesssim  \eps \lesssim \delta.\end{equation}
With a doubling scale found, we now choose the averaging scales.

\subsection{Choosing the averaging scales}  Fix $R>0$. Choose $\nu\in \SK$ with $\alpha_{\mu, \nu}(B(x, Mr_0))<\alpha.$  Provided $\alpha$ is sufficiently small in terms of $\eps$ and $M$, then Alternative \ref{alt} tells us that either
\begin{enumerate}
\item there exists $\wt x \in B(x,r_0)$ with $\nu$ reflection symmetric in $B(\wt x,32 r_0)$, and we set $R=1$, or there exists $\wt x \in B(x, (32 \Theta +1)r_0)$ such that $\nu$ is reflection symmetric in $B(\wt x, 64(32 \Theta)r_0)$,  in which case we determine $R=2^6\Theta = 2 \cdot 32 \Theta \geq 32 \Theta + 1$.
\item $\mu(B(x, 2^{50}\Theta r_0))\lesssim \eps  r_0^s$, in which case we set $\wt x=x$ and $R=1$.
\end{enumerate}

Notice that, in either case $\wt x \in B(x, Rr_0)$ (we say this with a view to applying Corollary \ref{corlong}, Proposition \ref{propinter}, and Proposition \ref{propunit}, with $x_0=\wt x$ and $r=Rr_0$).

With this notation, we shall prove the following proposition:

\begin{cla}\label{avecomp}  One has
$$\Bigl|T_{ r_0}(\mu)(x) - \frac{1}{\sigma}\int_4^8\int_{B(\wt x,   Q R  r_0)}T(\chi_{B(\wt x,  Q R r_0)^c}\, \mu)(y) d\mu(y) dQ\Bigl|\lesssim \delta,
$$
where $\sigma=\int_{4}^{8} \mu(B(\wt x, Q R  r_0))\,dQ$.
\end{cla}
\begin{proof}
Fix $M' \in [M+2^{10}, 2^{11}M]$ according to Lemma \ref{thinshell} with $r=Rr_0$, so $\mu(B(x, (M'+2^{10})Rr_0)\backslash B(x, (M'-2^{10})Rr_0))\leq \frac{2}{M}\mu(B(x, 2^{11}Mr_0))$.  We estimate the difference appearing on the left hand side of Claim \ref{avecomp} by
\begin{align*}
& \Bigl|T_{r_0}(\mu)(x) -\frac{1}{\sigma}\int_4^8 \int_{B(\wt{x}, QRr_0)}\!\!\!\!T([1-\eta_{\frac{1}{M'},M'r_0,y}]\mu)(y)d\mu(y)\,dQ \Bigl|\\
    & + \Bigl|\frac{1}{\sigma}\int_{4}^{8} \int_{B(\wt{x}, QRr_0)} T ([\eta_{\frac{1}{M'},M'r_0,y}-\eta_{\frac{1}{16},16Rr_0,y}]\mu)(y) \,d\mu(y) \,dQ\Bigl| \\
    & +\Bigl|\frac{1}{\sigma} \int_{4}^{8}  \int_{B(\wt{x}, QRr_0)} T ([\eta_{\frac{1}{16},16Rr_0,y} \,\,\chi_{B(\wt{x}, QRr_0)^c}]\mu)(y) \,d\mu(y) \,dQ\Bigl|\\
    &=  I + II + III.
    \end{align*}
Appealing Corollary \ref{corlong}, with $x_0=\wt{x}$ to the scale $r= Rr_0$ with $t=\frac{r}{R}$ (which is much greater than $\frac{Rr_0}{M})$, we obtain for any $Q\in [4,8]$,
\begin{equation}\begin{split}\nonumber
\Bigl|\mu(B(\wt x, QRr_0)) T_{r_0}(\mu)(x) &- \int_{B(\wt x, QRr_0)}\!\!\!\!T([1-\eta_{\frac{1}{M'},M'r_0,y}]\mu)(y)d\mu(y)\Bigl| \\\leq& \delta \, \mu(B(\wt x, QRr_0)),
\end{split}\end{equation}
and hence, by integrating both sides of this inequality over $Q\in [4,8]$, we see that $I\leq \delta$.

To bound $II$, we appeal to Proposition \ref{propinter} with $x_0=\tilde{x}$, $M$ replaced by $M'$ and $r=Rr_0$. This yields
\begin{equation}\begin{split} \nonumber \Bigl|\int_{B(\wt x, QRr_0)}&T([\eta_{\frac{1}{M'},M'r_0,y}- \eta_{\frac{1}{16}, 16R r_0,y}]\mu)d\mu(y)\Bigl| \leq \delta \, \mu(B(\wt x, QRr_0)),\end{split}\end{equation}
and integrating this inequality over $Q\in [4,8]$ we get that $II\leq \delta$.

Lastly, we turn to $III$, for which we intend to apply Proposition \ref{propunit}. With $x_0=\tilde{x}$ and $r=Rr_0$, either alternative (1) or (2) is satisfied in the hypothesis of Proposition \ref{propunit} with the measure $\nu$.  Moreover $\alpha_{\mu,\nu}(B(x, \gamma Rr_0))<\alpha$ with $\gamma  = \frac{1}{R}$ (which depends on $\Theta$), so $\gamma M\gg 1$. Consequently, we may apply Proposition \ref{propunit} to get that
\begin{equation} \begin{split}\nonumber\Bigl|\int_{4}^{8}\int_{B(x_0, QR r_0)}  T&(\eta_{\frac{1}{16}, 16R r_0,y}\;\chi_{B(\wt x, QRr_0)^c}\mu)d\mu(y)dQ\Bigl|<\delta \mu(B(\wt x,2^{10}R r_0)).\end{split}\end{equation}
But now $B(\wt x, 2^{10}Rr_0)\subset B(x, Ar_0)$, and $B(x, r_0)\subset B(\wt x, 4Rr_0)$, so the doubling property, and monotonicity of the measure, yields $$\mu(B(\wt x, 2^{10}Rr_0))\lesssim \mu(B(x, r_0))\lesssim \int_4^8\mu(B(x, QRr_0))dQ=\sigma.$$
Combining these observations yields $III\lesssim \delta$, and so the claim is proved.
\end{proof}
To finish the proof Theorem \ref{GenThm}, notice that $x\in \frac{1}{2}B(\wt x,  Q R r_0)$ for every $Q\in [4,8]$, and so (\ref{epsdiff}) ensures that
$$\Bigl|\wt T(\mu)(x) - \frac{1}{\sigma}\int_4^8\int_{B(\wt x, QR r_0)} T(\chi_{B(\wt x,QR r_0)^c}\mu)(y) d\mu(y) dQ\Bigl|\lesssim \delta.$$
Therefore from Claim \ref{avecomp} we get that
$$|\wt T(\mu)(x) - T_{r_0}(\mu)(x)|\lesssim \delta.
$$
So applying (\ref{nondoubnoprob2}),
$$|\wt T(\mu)(x) - T_r(\mu)(x)| \lesssim \delta.
$$
Therefore, $\lim_{r\to 0}T_r(\mu)(x) = \wt T(\mu)(x)$, and Theorem \ref{GenThm} is proved.

\appendix
\section{\\The Mattila-Verdera Weak Limit}
For the reader's convenience, in this section we provide the proof of Theorem \ref{Weak}.
In order to do so, we recall the following version of the Lebesgue Differentiation Theorem:
\begin{thm}\label{LebDiff}
Let $\mu$ be a locally finite Borel measure and let $f$ be a locally integrable function. Then
$$f(x) =  \lim_{\substack{|B_r| \rightarrow 0\\ x \in \frac{1}{2}B_r}} \frac{1}{\mu(B_r)} \int_{B_r} f\,d\mu \text{    for  }\mu\text{-a.e. } x \in \mathbb{R}^d.$$
\end{thm}

This theorem follows in a standard manner from the weak (1,1) inequality for the Maximal funtion
$$M(f)(x) = \sup_{B\text{ a ball},\, x\in \tfrac{1}{2}B}\frac{1}{\mu(B)}\int_B f\, d\mu.
$$
Again appealing to standard theory, the weak-type inequality follows (for instance) from the validity of the following variant of the Besicovitch covering lemma:

\begin{lemma} Suppose that $\{B_j\}_j$ is a finite collection of balls, then there is a sub-collection $\{B_{j_k}\}_k$ such that
\begin{itemize}
\item there is a dimensional constant $M=M(d)>0$, such that for any $x\in \R^d$, $\operatorname{card}\{k: x\in B_{j_k}\}\leq M$, and
\item $\bigcup_k B_{j_k}\supset \bigcup_j \frac{1}{2}B_j.$
\end{itemize}
\end{lemma}

The verification of this lemma is an exercise in the proof of Besicovitch covering theorem, and is left to the reader (see also \cite[p.6-7]{deGuz} for more general statements).\\

Now we proceed with the proof of Theorem \ref{Weak}, which is a special case of the analysis in \cite{MV}.
\begin{proof}[Proof of Theorem \ref{Weak}]
We need to show that $T_{r}(\mu)$ has a weak limit in $L^2(\mu)$ as $r \rightarrow 0$. Due to the antisymmetry of $K$, for every $r >0$ and any open ball $B$ we have
$$\int_B T_{r}( \chi_B \mu) \,d \mu =0.$$
Consequently,
$$ \int_B T_{r}(\mu) \,d \mu = \int_B T_{r} (\chi_{\R^d\backslash B}\mu) \,d \mu.$$
Notice that $T(\chi_{\R^d\backslash B}\mu)(x)$ is well-defined for $x \in B$. Also, since $B$ is open, $T(\chi_{\R^d\backslash B}\mu)(x) =  \lim_{r \rightarrow 0} T_{r} (\chi_{\R^d\backslash B}\mu)(x),$ for $x \in B$.
Moreover,
$$|T_{r} (\chi_{\R^d\backslash B}\mu)(x)| \leq \sup_{r >0} |T_{r}(\chi_{\R^d\backslash B}\mu)(x)| \stackrel{\text{def}}{=} T_{\ast} (\chi_{\R^d\backslash B}\mu)(x).$$
It is well known that (within the class of convolution kernels we consider in Definition \ref{CZdef}) $L^2(\mu)$ boundedness of $T$ implies that $T_{\ast}(\mu)\in L^1(\mu)$ (see \cite{NTV3} and \cite{To6}). Consequently, from the dominated convergence theorem we deduce that
\begin{equation}\label{weaklimball}\lim_{r \rightarrow 0} \int_B T_{r} (\mu) \,d\mu = \int_B T ( \chi_{\R^d\backslash B}\mu )d\mu.\end{equation}
In order to prove the weak convergence  in $L^2(\mu)$ of $T_{r}(\mu) $ as $r \rightarrow 0$, we shall prove the existence of
$$\lim_{r \rightarrow 0} \int_{\R^d} T_{r} (\mu) g \,d\mu, \text{  for every }g \in L^{2}.$$

We know this limit exists if $g$ is the characteristic function of an open ball, and therefore the limit exists for linear combinations of characteristic functions of open balls, a collection of functions which we denote by $S$.  By appealing to (for instance) the Vitali covering lemma, one can see that $S$ is dense in $L^2(\mu)$.

Fix an arbitrary function $g \in L^{2}(\mu)$. For $\delta>0$, let $b \in S$ satisfy $\|g-b\|_{L^{2}(\mu)} < \delta$. Then, for $r_1, r_2 >0$,
\begin{align*}
 \int_{\R^d} (T_{r_1}(\mu) - T_{r_2} (\mu)) g \,d \mu= \int_{\R^d} & (T_{r_1}(\mu) - T_{r_2}(\mu)) b \,d \mu \\
 & + \int_{\R^d} (T_{r_1} (\mu) - T_{r_2} (\mu)) (g-b) \,d \mu.
\end{align*}
The second term is bounded by $2 \|T_{\ast}(\mu) \|_{L^2(\mu)}\|g-b\|_{L^{2}(\mu)} \leq 2 \delta  \|T_{\ast}(\mu) \|_{L^2(\mu)}$.
Consequently,
$$ \limsup_{r_1,r_2 \rightarrow 0} \left| \int_{\R^d} (T_{r_1}(\mu) - T_{r_2}(\mu) ) g \,d \mu \right| \leq 2 \delta \|T_{\ast}(\mu) \|_{L^2(\mu)}. $$
As $\delta>0$ is arbitrary, we have that $T_{r} (\mu)$ converges weakly in $L^2(\mu)$ as $r \rightarrow 0$ to a function $\widetilde{T}(\mu)$.

Now, by Theorem \ref{LebDiff}, for $\mu$-almost every $x \in \mathbb{R}^d$ we have that
\begin{align*}
 & \widetilde{T}(\mu) (x)= \lim_{\substack{|B_r| \rightarrow 0\\ x \in \frac{1}{2}B_r}} \frac{1}{\mu(B_r)} \int_{B_r} \widetilde{T}(\mu) \,d\mu \\ & \stackrel{\text{weak convergence}}{=} \lim_{\substack{|B_r| \rightarrow 0\\ x \in \frac{1}{2}B_r}}\lim_{t\to 0} \frac{1}{\mu(B_r)} \int_{B_r} T_t(\mu)\,d\mu\\
 & \stackrel{(\ref{weaklimball})}{=} \lim_{\substack{|B_r| \rightarrow 0\\ x \in \frac{1}{2}B_r}} \frac{1}{\mu(B_r)} \int_{B_r} T(\chi_{\R^d\backslash B_r}\mu)\,d\mu
 \end{align*}
 and the theorem is proved.
\end{proof}

\end{document}